\documentclass[12pt,a4paper]{amsart}
\usepackage{amsfonts,color}
\usepackage{amsthm}
\usepackage{amsmath}
\usepackage{amscd}
\usepackage[latin2]{inputenc}
\usepackage{t1enc}
\usepackage[mathscr]{eucal}
\usepackage{indentfirst}
\usepackage{graphicx}
\usepackage{graphics}
\usepackage{pict2e}
\usepackage{epic}
\usepackage{url}
\usepackage{comment}
\usepackage{longtable}

\numberwithin{equation}{section}
\usepackage[margin=2.6cm]{geometry}

 \usepackage{tikz}
\usepackage{tkz-graph}
\usetikzlibrary{arrows,shapes,snakes,automata,backgrounds,petri}
\tikzstyle{vertex}=[draw=black,circle,fill=black,minimum size=4pt, inner sep=0pt, outer sep=0pt,text=white,line width=0mm]
\tikzstyle{c0}=[shape=circle, minimum size=4pt, fill=white]
\tikzstyle{c1}=[shape=rectangle, minimum size=7pt, fill=red]
\tikzstyle{c2}=[shape=diamond, minimum size=10pt, fill=blue]

\theoremstyle{plain}
\newtheorem{Th}{Theorem}[section]
\newtheorem{Lemma}[Th]{Lemma}
\newtheorem{Cor}[Th]{Corollary}

 \theoremstyle{definition}
\newtheorem{Def}[Th]{Definition}
\newtheorem{Conj}[Th]{Conjecture}
\newtheorem{Rem}[Th]{Remark}
\newtheorem{?}[Th]{Problem}

\begin{document}

\title{Matchings in regular graphs: minimizing the partition function}

\author[M. Borb\'enyi]{M\'arton Borb\'enyi}
\address{ELTE: E\"{o}tv\"{o}s Lor\'{a}nd University\\H-1117 Budapest
\\ P\'{a}zm\'{a}ny P\'{e}ter s\'{e}t\'{a}ny 1/C}
\email{marton.borbenyi@gmail.com}

\author[P. Csikv\'ari]{P\'{e}ter Csikv\'{a}ri}

\address{MTA-ELTE Geometric and Algebraic Combinatorics Research Group \\ P\'{a}zm\'{a}ny P\'{e}ter s\'{e}t\'{a}ny 1/C \\ Hungary \& E\"{o}tv\"{o}s Lor\'{a}nd University \\ Mathematics Institute, Department of Computer 
Science \\ H-1117 Budapest
\\ P\'{a}zm\'{a}ny P\'{e}ter s\'{e}t\'{a}ny 1/C \\ Hungary \& Alfr\'ed R\'enyi Institute of Mathematics \\ H-1053 Budapest \\ Re\'altanoda utca 13-15} 

\email{peter.csikvari@gmail.com}

\thanks{The first author is partially supported by the EFOP program (EFOP-3.6.3-VEKOP-16-2017-00002). The first author was also partially supported be the New National Excellence Program (\'UNKP) when the project started.}



\begin{abstract} For a graph $G$ on $v(G)$ vertices let $m_k(G)$ denote the number of matchings of size $k$, and consider the partition function $M_{G}(\lambda)=\sum_{k=0}^nm_k(G)\lambda^k$. In this paper we show that if $G$ is a $d$--regular graph and $0<\lambda<(4d)^{-2}$, then
$$\frac{1}{v(G)}\ln M_G(\lambda)>\frac{1}{v(K_{d+1})}\ln M_{K_{d+1}}(\lambda).$$
The same inequality holds true if $d=3$ and $\lambda<0.3575$. More precise conjectures are also given.
\end{abstract}

\maketitle

\section{Introduction}

In this paper we study matchings of regular graphs. For a graph $G$ on $v(G)$ vertices let $m_k(G)$ denote the number of matchings of size $k$, and consider the function $M_{G}(\lambda)=\sum_{k=0}^nm_k(G)\lambda^k$. 
(Note that $m_0(G)=1$.) In mathematics this is called the matching generating polynomial, and in statistical physics it is called the partition function of the monomer-dimer model. Some authors also call it the matching polynomial of the graph $G$, but we reserve this name for another polynomial introduced later.

Thanks to the influential paper of Friedland, Krop and Markstr\"om \cite{friedland2008number} there has been a recent activity on extremal problems about matchings in regular graphs. General questions look as follows: fix the number of vertices $n$, the degree $d$ and the size $k$ of the matching, which $d$--regular graph maximizes or minimizes $m_k(G)$ among $d$--regular graphs on $n$ vertices? If one wishes to compare graphs of different sizes, then a good candidate for normalization is to consider $\frac{1}{v(G)}\ln M_G(\lambda)$. Then if we consider $k$ disjoint copies of the same graph $G$ -- let us denote it by $kG$-- then 
$$\frac{1}{v(kG)}\ln M_{kG}(\lambda)=\frac{1}{v(G)}\ln M_G(\lambda).$$
Generally, answering this latter question give an asymptotic answer to the former one by choosing $\lambda$ appropriately, and this way making the term $m_k(G)\lambda^k$ dominant in $M_G(\lambda)$.

A general intuition behind the answer for these questions that short even cycles increase $m_k(G)$ or $\frac{1}{v(G)}\ln M_G(\lambda)$, while short odd cycles decrease these quantities. Friedland, Krop and Markstr\"om \cite{friedland2008number} suggested the following conjectures: if $2d\ |\ n$ and $G$ is a $d$--regular graph on $n$ vertices, then
$$m_k(G)\leq m_k\left(\frac{n}{2d}K_{d,d}\right),$$
where $K_{d,d}$ is the complete bipartite graph with parts of size $d$. This conjecture became known as the Upper Matching Conjecture. It has also an asymptotic version, the Asymptotic 
Upper Matching Conjecture asserting that for any $d$--regular graph $G$ we have
$$\frac{1}{v(G)}\ln M_G(\lambda)\leq \frac{1}{v(K_{d,d})}\ln M_{K_{d,d}}(\lambda).$$
(This is not the original formulation, but it is roughly equivalent with it.) The Asymptotic 
Upper Matching Conjecture was proved by Davies, Jenssen, Perkins and Roberts \cite{davies2017independent}. In another paper, Perkins, Jenssen \cite{davies2020proof} also settled the Upper Matching Conjecture in the case when $n$ is large enough. 
These conjectures have also counterparts about minimization  among bipartite regular graphs, the Lower Matching Conjecture and Asymptotic Lower Matching Conjecture. Instead of giving the exact forms of these conjectures, which are a bit technical, we just give the flavour of these conjectures. In bipartite graphs there are no odd cycles, so the above intuition suggests that all we have to do is to minimize the number of even short cycles. This means that the graph will locally look like a tree. And indeed, in this case, there is no finite minimizing graph, but the infinite $d$--regular tree will play the role of the minimizing graph. It turns out that this intuition is completely right. Despite the complicated forms of these conjectures, they turned out to be easier than the versions concerning the maximization problems.  The Asymptotic Lower Matching Conjecture was proved by Gurvits \cite{gurvits2011unleashing}, the Lower Matching Conjecture was proved by Csikv\'ari \cite{csikvari2014lower}. 

In case of non-bipartite graphs the intuition suggests that the graph minimizing $\frac{1}{v(G)}\ln M_G(\lambda)$ should contain a lot of short odd cycles, notably triangles, and this suggests $K_{d+1}$, the complete graph on $d+1$ vertices, to be the minimizing graph. The problem is that $K_{d+1}$ also contains a lot of short even cycles which makes the situation very unclear. Checking various graphs with computer suggests the following conjectures:

\begin{Conj} \label{main conjecture}
Let $G$ be a $d$--regular graph. \\
(a) If $d$ is even, then for all $\lambda\geq 0$ we have
$$\frac{1}{v(G)}\ln M_G(\lambda)\geq \frac{1}{v(K_{d+1})}\ln M_{K_{d+1}}(\lambda).$$
(b) If $d$ is odd, then there exists a constant $c_d$ defined later such that for $0\leq \lambda\leq c_d$ we have
$$\frac{1}{v(G)}\ln M_G(\lambda)\geq \frac{1}{v(K_{d+1})}\ln M_{K_{d+1}}(\lambda),$$
and for $\lambda>c_d$, the graph $K_{d+1}$ is never the minimizer graph.
\end{Conj}

The constants $c_d$ will be defined in Section~\ref{covers and necklaces}. At this moment we just remark that $c_3=1$ and $c_3<c_5<c_7<\dots$ and $\lim_{d\to \infty}c_d=\infty$.

Concerning this conjecture Davies, Perkins and Jenssen \cite{davies2020proof} proved the following result.

\begin{Th}[Davies, Perkins and Jenssen \cite{davies2020proof}]
There exists an absolute constant $c$ such that if $0\leq \lambda<cd^{-4}$ and $G$ is a $d$--regular graph, then 
$$\frac{1}{v(G)}\ln M_G(\lambda)\geq \frac{1}{v(K_{d+1})}\ln M_{K_{d+1}}(\lambda).$$
\end{Th}

In this paper we improve on this result as follows.

\begin{Th} \label{main}
If $0\leq \lambda<(4d)^{-2}$  and $G$ is a $d$--regular graph, then 
$$\frac{1}{v(G)}\ln M_G(\lambda)\geq \frac{1}{v(K_{d+1})}\ln M_{K_{d+1}}(\lambda).$$
\end{Th}

This theorem is still very very far from the conjecture. 
\bigskip

As we will see in the proof of Theorem~\ref{main} a natural barrier of the argument is $1/(4(d-1))$. For $d=3$ this is $0.125$. Still we will be able to overcome this barrier. In particular, we will prove the following result.

\begin{Th} \label{3-reg main}
If $G$ is a $3$--regular graph and $0<\lambda<0.3575$, then
$$\frac{1}{v(G)}\ln M_G(\lambda)\geq \frac{1}{v(K_{4})}\ln M_{K_{4}}(\lambda).$$
\end{Th}

Concerning the case of even $d$ we believe that 
$$\frac{1}{v(G)}\ln M_G(\lambda)\geq \frac{1}{v(K_{d+1})}\ln M_{K_{d+1}}(\lambda)$$
holds true for all $\lambda$ and $d$--regular graph $G$. 
We prove that this is indeed the case if $\lambda$ is large enough.

\begin{Th} \label{even d}
Suppose that $d$ is even. Then there exists a $C(d)$ such that if $\lambda>C(d)$, then for every $d$--regular graph $G$ we have
$$\frac{1}{v(G)}\ln M_G(\lambda)\geq \frac{1}{v(K_{d+1})}\ln M_{K_{d+1}}(\lambda).$$
\end{Th}

\bigskip

\noindent \textbf{Notations.} If $G=(V,E)$ is a graph with vertex set $V(G)$ and edge set $E(G)$, then $v(G)$ denotes the number of vertices, and $e(G)$ denotes the number of edges. $K_n$ denotes the complete graph on $n$ vertices, and $C_n$ is the cycle on $n$ vertices. The infinite $d$--regular tree will be denoted by $\mathbb{T}_d$. We will denote by $D$ the diamond graph, the unique graph on $4$ vertices and $5$ edges.

For a vertex $v\in V(G)$ the set of neighbors of $v$ will be denoted by $N_G(v)$. If the graph is clear from the context, then we omit $G$ from the subscript. The degree of a vertex is denoted by $d_v$. The largest degree of a graph $G$ is denoted by $\Delta(G)$. If $G$ is clear from the context, then we simply write $\Delta$.

The set of all matchings is denoted by $\mathcal{M}(G)$.
The number of matchings of size $k$ is denoted by $m_k(G)$.
The size of a largest matching will be denoted by $\nu=\nu(G)$.  The partition function is defined by
$$M_G(\lambda):=\sum_{M\in \mathcal{M}}\lambda^{|M|}=\sum_{k=0}^{\nu(G)}m_k(G)\lambda^k.$$
We will also use the notation $M(G,\lambda)$ for $M_G(\lambda)$ when we wish to emphasize the role of $G$ more. For instance, when we write up a recursion for $M(G,\lambda)$ or study some special graph.

If $G$ and $H$ are graphs, then $N(H,G)$ denotes the number of subgraphs of $G$ isomorphic to $H$. For instance, $N(C_4,K_4)=3$. We will also  use the notation
$$\rho(H,G)=\frac{N(H,G)}{v(G)}$$
and $\rho_k=\rho(C_k,G)$. 

Additional notations and concepts will be introduced in the text. 
\bigskip

\noindent \textbf{This paper is organized as follows.} In the next section we collect a few tools that we will use throughout the paper. In Section~\ref{general d} we will prove the first part of Theorem~\ref{main}. In Section~\ref{3-regular} we will prove the second part of Theorem~\ref{main}. In Section~\ref{covers and necklaces} we will introduce a graph class that beats $K_{d+1}$ for odd $d$ and large $\lambda$. In particular we will introduce the constants $c_d$ of Conjecture~\ref{main conjecture} in this section. In Section~\ref{sec: even d} we prove Theorem~\ref{even d}.

\section{Preliminaries} \label{preliminaries}

Many of the things we we wish to use are in terms of the so-called matching polynomial. The matching polynomial is defined as follows: let $v(G)=n$ and
$$\mu(G,x)=\sum_{k=0}^{\lfloor n/2\rfloor}(-1)^km_k(G)x^{n-2k}.$$
Clearly, $M_G(\lambda)$ and $\mu(G,x)$ encode the same information. The matching polynomial $\mu(G,x)$ is an even or odd function depending on the parity of $v(G)$. Consequently,
the zeros of $\mu(G,x)$ are symmetric to $0$. In fact, the zeros are all real as the following theorem of Heilmann and Lieb shows.

\begin{Th}[Heilmann and Lieb \cite{heilmann1972theory}] \label{th: HeiLie}
All zeros of the matching polynomial $\mu(G,x)$ are real. Furthermore, if the largest degree $\Delta$ satisfies $\Delta \geq 2$, then all zeros lie in the interval $(-2\sqrt{\Delta-1},2\sqrt{\Delta-1})$.
\end{Th}

Let $\mu(G,x)=\prod_{i=1}^n(x-\alpha_i)$. Since it is an even or odd function we have
$$\mu(G,x)=x^c\prod_{\alpha_i>0}(x-\alpha_i)(x+\alpha_i),$$
where $c$ is the multiplicity of $0$ as a zero. Then
$$M_G(\lambda)=\prod_{\alpha_i>0}(1+\alpha_i^2\lambda).$$
It is convenient to introduce the notation $\gamma_i=\alpha_i^2$ for $\alpha_i>0$. Then $M_G(\lambda)=\prod (1+\gamma_i\lambda)$.
We will also need 
$$a_k(G)=\frac{1}{v(G)}\sum_{i=1}^{\nu} \gamma_i^k.$$
Unfortunately, in the literature many results are in terms of the quantities
$$s_k(G)=\frac{1}{v(G)}\sum_{i=1}^n \alpha_i^k.$$
Clearly, $a_k(G)=s_{2k}(G)/2$. The term $1/2$ comes from the fact that the non-zero roots of $\mu(G,x)$ come in pairs $\pm \alpha_i$. To understand the quantities $a_k(G)$ and $s_k(G)$ we need the concept of path-tree. 

\begin{Def} Let $G$ be graph with a given vertex $u$. The \textit{path-tree} $T(G,u)$ is defined as follows. The vertices of $T(G,u)$ are the paths in $G$ which start at the  vertex $u$ and two paths joined by an edge if one of them is a one-step extension of the other.
\end{Def}

\begin{figure}
\begin{center}
\includegraphics[scale=0.16]{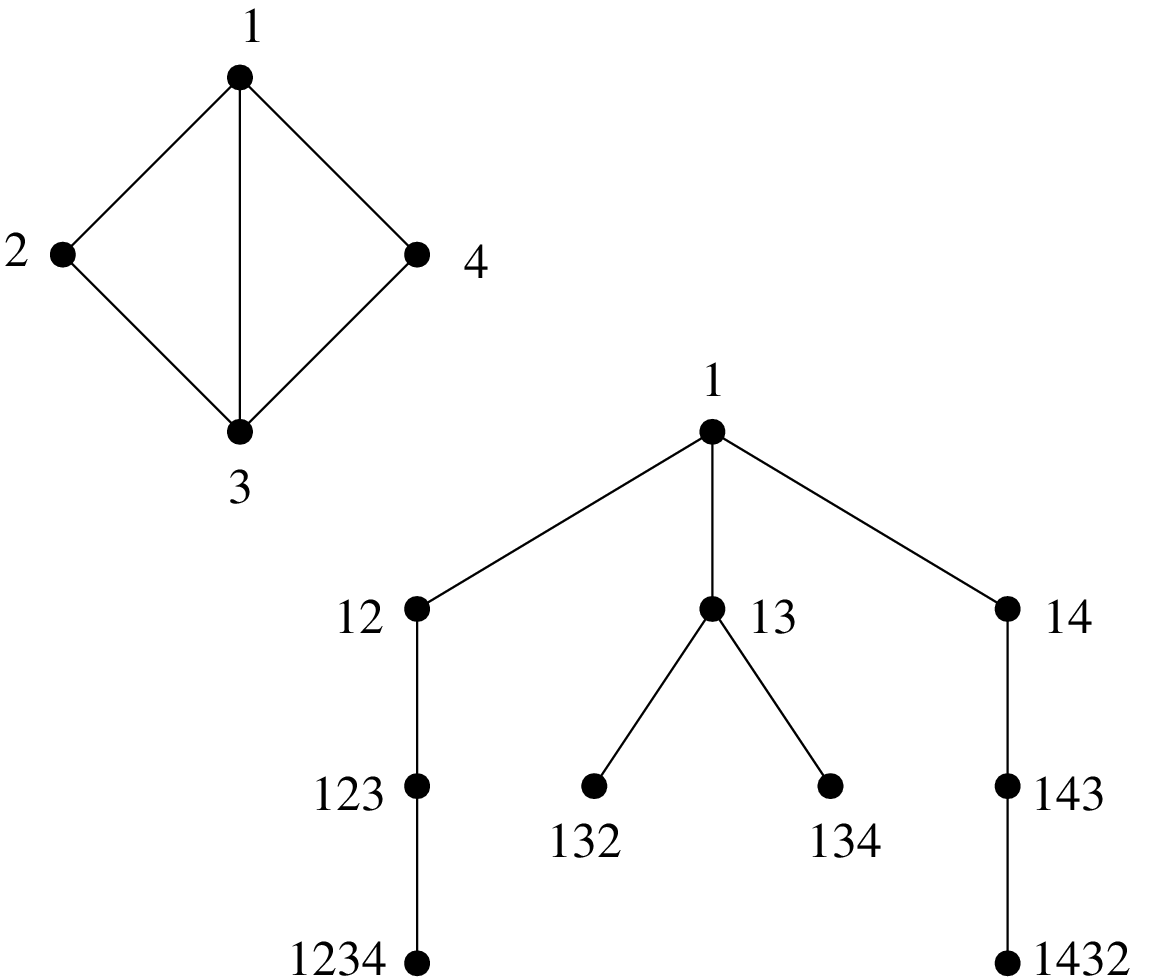}
\caption{A path-tree from the vertex $1$.}  
\end{center}
\end{figure}

It turns out that $s_k(G)$ counts certain walks called tree-like walks.

\begin{Def}
A closed tree-like walk of length $\ell$ is a closed walk on $T(G,u)$ of length $\ell$ starting and ending at the vertex $u$.
\end{Def}

Note that a priory the tree-like walk is on the path-tree, although one can make a correspondence with certain walks on the graph itself. Indeed, a walk in the tree $T(G,u)$ from $u$ can be imagined as follows. Suppose that in the graph $G$ a worm is sitting at the vertex $u$ at the beginning. Then at each step the worm can either grow or pull back its head. When it grows it can move its head to a neighboring unoccupied vertex while keeping its tail at vertex $u$.  At each step the worm occupies a path in the graph $G$. A closed walk in the tree $T(G,u)$ from $u$ to $u$ corresponds to the case when at the final step the worm occupies only vertex $u$.  

After this long introduction to tree-like walks the following theorem is not surprising.

\begin{Lemma}[Godsil \cite{godsil1993algebraic}] \label{tree-like walks}
The number of all closed tree-like walks of length $\ell$ is exactly $s_{\ell}(G)$.
\end{Lemma}

Note that we can introduce $s_k(\mathbb{T}_d)$ this way: this is simply the number of closed walks from a root vertex $u$ of the infinite $d$--regular tree of length $k$. Then $a_k(\mathbb{T}_d)=s_{2k}(\mathbb{T}_d)/2$.

An immediate corollary of Lemma~\ref{tree-like walks} is the following lemma. A variant of this statement already appeared in the work of McKay and Wanless \cite{mckay1998maximising}.

\begin{Lemma} \label{comparison} For every $d$--regular graph $G$ and $k\geq 1$ we have 
$$a_k(K_{d+1})\leq a_k(G)\leq a_k(\mathbb{T}_d).$$
\end{Lemma}

\begin{proof}
The path-tree of the complete graph is a subtree of the path-tree of any regular graph. Indeed, if we have  a path $P$ in $G$ consisting of $r$ vertices, then we can extend it to at least $d+1-r$ ways to a path on $r+1$ vertices. In $K_{d+1}$ we have equality here. 

Similarly, the path-tree of any regular graph is a subtree of the infinite $d$--regular tree.

\end{proof}

\noindent If $d=3$ we have the following tables. The graphs $\mathrm{DN}_3$ and $\mathrm{DN}_2$ are defined in Section~\ref{covers and necklaces}.
\bigskip 

$$\begin{array}{|c|c|c|c|c|c|c|c|c|c|c|} \hline
k & 1 & 2 & 3 & 4 & 5 & 6 & 7 & 8 & 9 & 10 \\ \hline
2a_k(K_4) & 3 & 15 & 81 & 441 & 2403 & 13095 & 71361 & 388881 & 2119203 & 11548575 \\ \hline  
2a_k(\mathbb{T}_3) & 3 & 15 & 87 & 543 & 3543 & 23823 & 163719 & 1143999 & 8099511 & 57959535 \\ \hline
2a_k(\mathrm{DN}_3) & 3 & 15 & 84 & 493 & 2973 & 18261 & 113676 & 714849 & 4530843 & 28897155 \\ \hline
2a_k(\mathrm{DN}_2) & 3 & 15 & 84 & 493 & 2973 & 18255 & 113494 & 711673 & 4488663 & 28422175 \\ \hline
\end{array}$$
\bigskip

\noindent It is also clear from path-tree approach that if we know sufficiently large neighbors of the vertices, then we can determine $a_k(G)$. In other words, $a_k(G)$ can be expressed by small subgraph counts. We will need a few of these expressions.

\begin{Lemma}[Wanless \cite{wanless2010counting}] \label{densities}
For a $d$--regular graph $G$ we have\\
(a) $a_1(G)=a_1(\mathbb{T}_d)$,\\
(b) $a_2(G)=a_2(\mathbb{T}_d)$,\\
(c) $a_3(G)=a_3(\mathbb{T}_d)-3\rho(C_3,G)$,\\
(d) $a_4(G)=a_4(\mathbb{T}_d)-24(d-1)\rho(C_3,G)-4\rho(C_4,G)$,\\
(e) $a_5(G)=a_5(\mathbb{T}_d)-135(d-1)^2\rho(C_3,G)-40(d-1)\rho(C_4,G)-5\rho(C_5,G)+20\rho(D,G)$.
\end{Lemma}

We will  only use the first $4$ claims in the proof of Theorem~\ref{main}. So we will mainly deal with triangles and $4$-cycles.

\subsection{Approaches}

For general $d$ we will do the most basic approach one can use: we write up the Taylor-series of $\frac{1}{v(G)}\ln M_G(\lambda)$ and $\frac{1}{v(K_{d+1})}\ln M_{K_{d+1}}(\lambda)$, and compare them. As we will see shortly, this can be done if $|\lambda \gamma_i|<1$. Since $\gamma_i\in [0,4(d-1)]$ by the Heilmann-Lieb theorem this approach works if $|\lambda|<\frac{1}{4(d-1)}$. The actual computation is the following.

\begin{align*}
\frac{1}{v(G)}\ln M_G(\lambda)&=\frac{1}{v(G)}\ln \left(\prod_{i=1}^{\nu}(1+\gamma_i\lambda)\right)\\
&=\frac{1}{v(G)}\sum_{i=1}^{\nu}\ln(1+\gamma_i\lambda)\\
&=\frac{1}{v(G)}\sum_{i=1}^{\nu}\sum_{k=1}^{\infty}(-1)^{k-1}\frac{(\gamma_i\lambda)^k}{k}\\
&=\frac{1}{v(G)}\sum_{k=1}^{\infty}\frac{(-1)^{k-1}\lambda^k}{k}\sum_{i=1}^{\nu}\gamma_i^k\\
&=\sum_{k=1}^{\infty}a_k(G)\frac{(-1)^{k-1}\lambda^k}{k}
\end{align*}
Note that this Taylor-series enable us to define $\frac{1}{v(\mathbb{T}_d)}\ln M_{\mathbb{T}_d}(\lambda)$ by
$$\frac{1}{v(\mathbb{T}_d)}\ln M_{\mathbb{T}_d}(\lambda)=\sum_{k=1}^{\infty}a_k(\mathbb{T}_d)\frac{(-1)^{k-1}\lambda^k}{k}.$$
In the proof of Theorem~\ref{main} we will truncate this sum as follows.
$$\frac{1}{v(G)}\ln M_G(\lambda)=\sum_{k=1}^{5}a_k(G)\frac{(-1)^{k-1}\lambda^k}{k}+\varepsilon_G(\lambda).$$
To bound $\varepsilon_G(\lambda)$ we will use the standard bound that for $x\in (0,1)$ we have
$0\leq \sum_{j=t}^{\infty}(-1)^{t-j}\frac{x^j}{j}\leq \frac{x^t}{t}$
since $\frac{x^j}{j}> \frac{x^{j+1}}{j+1}$. This gives that
\begin{align*}
-\varepsilon_G(\lambda)&=\sum_{k=6}^{\infty}a_k(G)\frac{(-1)^k\lambda^k}{k}
=\frac{1}{v(G)}\sum_{k=6}^{\infty}\sum_{i=1}^{\nu}(-1)^{k}\frac{(\gamma_i\lambda)^k}{k}
=\frac{1}{v(G)}\sum_{i=1}^{\nu}\sum_{k=6}^{\infty}(-1)^{k}\frac{(\gamma_i\lambda)^k}{k}\\
&\leq \frac{1}{v(G)}\sum_{i=1}^{\nu}\frac{(\gamma_i\lambda)^6}{6}\leq \frac{1}{v(G)}\sum_{i=1}^{\nu}\frac{(4(d-1)\lambda)^6}{6}\leq \frac{(4(d-1)\lambda)^6}{6}.
\end{align*}
Similarly, we get that $-\varepsilon_G(\lambda)\geq 0$.

The idea behind this truncation is that in order to compare $\frac{1}{v(G)}\ln M_G(\lambda)$ and \\ $\frac{1}{v(K_{d+1})}\ln M_{K_{d+1}}(\lambda)$ we only need to deal with triangles and $4$-cycles. We do not have to deal with diamonds and $5$-cycles, because we get $a_5(G)\geq a_5(K_{d+1})$ for free from Lemma~\ref{comparison}.

Note that the combination of Taylor-series with subgraph counts leads to a technique called Mayer expansion originated from statistical physics. This is exactly the method that was used in the paper of Davies, Perkins and Jenssen \cite{davies2020proof}. In this sense we do not do anything new, except that we truncate the Taylor-series later, and we have a slightly better understanding of the coefficients. (It is hard to read out the inequality $a_k(G)\geq a_k(K_{d+1})$ from the Mayer expansion.)

The limitation of this argument is that we clearly cannot go beyond $\frac{1}{4(d-1)}$ with $\lambda$. So for $d=3$ we will use another argument, where we use a polynomial approximation of $\ln(1+\lambda x)$ on the interval $(0,4(d-1))$ instead of the Taylor-expansion. 

\section{Proof of Theorem~\ref{main}} \label{general d}

In this section we prove  Theorem~\ref{main}. The following lemma gives some inequalities between the density of triangles and $4$-cycles.

\begin{Lemma} \label{main-lemma}
Let $G$ be a $d$--regular graph without a connected component isomorphic to $K_{d+1}$. Then\\
(a)  $$\rho(C_3,K_{d+1})-\rho(C_3,G)\geq \frac{1}{3}.$$
(b) $$\rho(C_4,K_{d+1})-\rho(C_4,G)\leq \frac{3(d-2)}{2}(\rho(C_3,K_{d+1})-\rho(C_3,G)).$$
\end{Lemma}

\begin{Rem}
We remark that much more precise results about the density of triangles and $4$-cycles were achieved by Harangi, see \cite{harangi2013density}.
\end{Rem}

\begin{proof} For a vertex $v\in V(G)$ let $N(v)$ be the set of neighbors, and $t_v=\binom{d}{2}-e(N(v))$, that is, the number of non-edges in the neighbor of $v$. Then
\begin{align*}
\frac{N(C_3,G)}{v(G)}&=\frac{1}{3v(G)}\sum_{v\in V(G)}e(N(v))\\
        &=\frac{1}{3v(G)}\sum_{v\in V(G)}\left(\binom{d}{2}-t_v\right)\\
				&=\frac{1}{3}\binom{d}{2}-\frac{1}{3v(G)}\sum_{v\in V(G)}t_v\\
				&=\frac{N(C_3,K_{d+1})}{v(K_{d+1})}-\frac{1}{3v(G)}\sum_{v\in V(G)}t_v.
\end{align*}
Hence
$$\frac{N(C_3,K_{d+1})}{v(K_{d+1})}-\frac{N(C_3,G)}{v(G)}=\frac{1}{3v(G)}\sum_{v\in V(G)}t_v.$$
Now the first claim follows immediately since $t_v\geq 1$ for every vertex $v$.

Let us prove part (b).
When we count the number of $4$-cycles we first choose a $v=v_1$, then $v_2,v_3,v_4\in N(v)$ in such a way that $(v_2,v_3)\in E(G)$ and $(v_3,v_4)\in E(G)$. First of all, we do not count at all those $4$--cycles this way, where $v_3$ is not a neighbor of $v$. In general, if $v_1v_2v_3v_4$ is a $4$-cycle, then we count it $8$ times if both $(v_1,v_3)\in E(G)$ and $(v_2,v_4)\in E(G)$, we count it exactly $4$ times if exactly one of them is an edge. Observe also that a missing edge in $N(v)$ prevent at most $4(d-2)$ labeled triples $(v_2,v_3,v_4)$: if $(v_2,v_3)\notin E(G)$, then we can choose $v_4$ in $d-2$ ways, and a non-edge also prevents $(v_3,v_4)\notin E(G)$
and we have to multiply it by $2$, because of the labeling. Hence
\begin{align*}
\frac{N(C_4,G)}{v(G)}&\geq \frac{1}{8v(G)}\sum_{v\in V(G)}(d(d-1)(d-2)-4t_v(d-2))\\
                     &=\frac{1}{8}d(d-1)(d-2)-\frac{d-2}{2v(G)}\sum_{v\in V(G)}t_v\\
										&=\frac{N(C_4,K_{d+1})}{v(K_{d+1})}-\frac{d-2}{2v(G)}\sum_{v\in V(G)}t_v.
\end{align*}
Then
$$\frac{N(C_4,K_{d+1})}{v(K_{d+1})}-\frac{N(C_4,G)}{v(G)}\leq \frac{d-2}{2v(G)}\sum_{v\in V(G)}t_v=\frac{3(d-2)}{2}\left(\frac{N(C_3,K_{d+1})}{v(K_{d+1})}-\frac{N(C_3,G)}{v(G)}\right).$$

\end{proof}

\begin{proof}[Proof of Theorem~\ref{main}]
We have
$$\frac{1}{v(G)}\ln M_G(\lambda)=\sum_{k=1}^5a_k(G)\frac{(-1)^{k-1}\lambda^k}{k}+\varepsilon_G(\lambda),$$
where
$$0\leq -\varepsilon_G(\lambda)\leq \frac{1}{6}(4(d-1)\lambda)^6.$$
Let $S=\frac{1}{v(G)}\ln M_G(\lambda)-\frac{1}{v(K_{d+1})}\ln M_{K_{d+1}}(\lambda)$. 
Then
\begin{align*}
S&=\sum_{k=1}^5(a_k(G)-a_k(K_{d+1}))\frac{(-1)^{k-1}\lambda^k}{k}+\varepsilon_G(\lambda)-\varepsilon_{K_{d+1}}(\lambda) \\
 &=(a_3(G)-a_3(K_{d+1}))\frac{\lambda^3}{3}-(a_4(G)-a_4(K_{d+1}))\frac{\lambda^4}{4}+(a_5(G)-a_5(K_{d+1}))\frac{\lambda^5}{5}+\varepsilon_G(\lambda)-\varepsilon_{K_{d+1}}(\lambda) \\
 &\geq (a_3(G)-a_3(K_{d+1}))\frac{\lambda^3}{3}-(a_4(G)-a_4(K_{d+1}))\frac{\lambda^4}{4}+\varepsilon_G(\lambda)-\varepsilon_{K_{d+1}}(\lambda)\\
 &\geq (a_3(G)-a_3(K_{d+1}))\frac{\lambda^3}{3}-(a_4(G)-a_4(K_{d+1}))\frac{\lambda^4}{4}-\frac{1}{6}(4(d-1)\lambda)^6\\
 &=\left(\rho(C_3,K_{d+1})-\rho(C_3,G)\right)\left(3\frac{\lambda^3}{3}-24(d-1)\frac{\lambda^4}{4}\right)
   -\left(\rho(C_4,K_{d+1})-\rho(C_4,G)\right)\cdot 4\frac{\lambda^4}{4}\\
 &\ -\frac{1}{6}(4(d-1)\lambda)^6\\
 &\geq \left(\rho(C_3,K_{d+1})-\rho(C_3,G)\right)\left(3\frac{\lambda^3}{3}-24(d-1)\frac{\lambda^4}{4}-\frac{3(d-2)}{2}\cdot 4\cdot \frac{\lambda^4}{4}\right)  -\frac{1}{6}(4(d-1)\lambda)^6\\
 &\geq \frac{1}{3}\left(3\frac{\lambda^3}{3}-24(d-1)\frac{\lambda^4}{4}-\frac{3(d-2)}{2}\cdot 4\cdot \frac{\lambda^4}{4}\right)  -\frac{1}{6}(4(d-1)\lambda)^6\\
 &\geq \frac{1}{6}( 2\lambda^3-15d\lambda^4-(4d\lambda)^6).
\end{align*}
In the first step we used that $a_5(G)\geq a_5(K_{d+1})$, then we used that\\  $\varepsilon_G(\lambda)\geq -\frac{1}{6}(4(d-1)\lambda)^6$ and $-\varepsilon_{K_{d+1}}(\lambda)\geq 0$. Then we expressed  $a_3(G)$ and $a_4(G)$ in terms of subgraph densities via Lemma~\ref{densities}. Then we used part (b) of Lemma~\ref{main-lemma}, and then part (a) of the same lemma.
The last term is positive if $\lambda^{-1}>16d^2$ since then $\lambda^3>(4d\lambda)^6$ and $\lambda^3>15d\lambda^4$.

\end{proof}

\subsection{Negative values}

In this section we study the problem for negative $\lambda$.

\begin{Th} Let $-\frac{1}{4(d-1)}\leq \lambda \leq 0$. Then
$$\frac{1}{v(\mathbb{T}_d)}\ln M_{\mathbb{T}_d}(\lambda)\leq \frac{1}{v(G)}\ln M_G(\lambda)\leq \frac{1}{v(K_{d+1})}\ln M_{K_{d+1}}(\lambda).$$
\end{Th}

\begin{proof}
Let $\lambda=-x$, then $0\leq x\leq \frac{1}{4(d-1)}$, and 
\begin{align*}
\frac{1}{v(G)}\ln M_G(-x)&=\frac{1}{v(G)}\sum_{i=1}^{v(G)/2}\ln(1-\gamma_i(G)x)\\
&=\frac{1}{v(G)}\sum_{i=1}^{\nu(G)}\sum_{k=1}^{\infty}\frac{-1}{k}(\gamma_i(G)x)^k\\
&=-\sum_{k=1}^{\infty}\frac{a_k(G)}{k}x^k.
\end{align*}
Now the claim follows from the inequalities $a_k(K_{d+1})\leq a_k(G)\leq a_k(\mathbb{T}_d)$.
\end{proof}

\begin{Rem}
We have
$$\frac{1}{v(\mathbb{T}_d)}\ln M(\mathbb{T}_d,\lambda)=\frac{1}{2}\ln S_d(\lambda),$$
where 
$$S_d(\lambda)=\frac{1}{\eta_{\lambda}^2}\left(\frac{d-1}{d-\eta_{\lambda}}\right)^{d-2}\ \ \ \mbox{and}\ \ \ \eta_{\lambda}=\frac{\sqrt{1+4(d-1)\lambda}-1}{2(d-1)\lambda}.$$

\end{Rem}

\section{3--regular graphs} \label{3-regular}

In this section we prove Theorem~\ref{3-reg main}. The main ingredient is the following technical lemma.

\begin{Th}
Assume that the polynomial $P(x)=c_0+c_1x+c_2x^2+c_3x^3+c_4x^4+c_5x^5$ satisfies that $(-1)^{i+1}c_i\geq 0$ for $3\leq i\leq 5$, and for all $x\in [0,A]$ we have
$$|\ln(1+x)-P(x)|\leq \varepsilon.$$
Furthermore, suppose that $\lambda\leq A/8$ and satisfies that
$$\frac{3}{2}c_3\lambda^3+27c_4\lambda^4-\varepsilon>0.$$
Then
$$\frac{1}{v(G)}\ln M_G(\lambda)> \frac{1}{4}\ln M_{K_4}(\lambda)$$
for every $3$-regular graph $G$, if $G$ is not disjoint union of $K_4$ graphs.
\end{Th} 

\begin{Rem}
Note that for a fix $A$ the inequality $\frac{3}{2}c_3\lambda^3+27c_4\lambda^4-\varepsilon>0$ gives an interval $(\lambda_{\min}(A),\lambda_{\max}(A))$ for which this is satisfied. Unfortunately, this only gives the interval $(\lambda_{\min}(A),\min(A/8,\lambda_{\max}(A)))$
for which
$$\frac{1}{v(G)}\ln M_G(\lambda)> \frac{1}{4}\ln M_{K_4}(\lambda).$$
Later we will use a ladder of $A$'s to cover a longer interval with intervals \\ $(\lambda_{\min}(A),\min(A/8,\lambda_{\max}(A)))$. Since $0$ is never in this interval we will use  Theorem~\ref{main} to handle the case of small $\lambda$. 
\end{Rem}

\begin{proof}

\begin{align*}
\frac{1}{v(G)}\ln M_G(\lambda)&=\frac{1}{2n}\sum_{i=1}^n \ln(1+\gamma_i(G)\lambda)\\ 
&\geq \frac{1}{2n}\sum_{i=1}^n(P(\gamma_i(G)\lambda)-\varepsilon)\\
&=\frac{1}{2}c_0+c_1a_1(G)\lambda+c_2a_2(G)\lambda^2+c_3a_3(G)\lambda^3+c_4a_4(G)\lambda^4+c_5a_5(G)\lambda^5-\frac{\varepsilon}{2}.
\end{align*}
The inequality holds true since $\gamma_i(G)\lambda\leq 8\lambda\leq A$.
Similarly, we have
\begin{align*}
\frac{1}{4}\ln M_{K_4}(\lambda)&=\frac{1}{4}\sum_{i=1}^2 \ln(1+\gamma_i(K_4)\lambda)\\ 
&\leq \frac{1}{4}\sum_{i=1}^2(P(\gamma_i(K_4)\lambda)+\varepsilon)\\
&=\frac{1}{2}c_0+c_1a_1(K_4)\lambda+c_2a_2(K_4)\lambda^2+c_3a_3(K_4)\lambda^3+c_4a_4(K_4)\lambda^4+c_5a_5(K_4)\lambda^5+\frac{\varepsilon}{2}.
\end{align*}
Hence
$$\frac{1}{v(G)}\ln M_G(\lambda)-\frac{1}{4}\ln M_{K_4}(\lambda)\geq \ \ \ \ \ $$
$$\ \ \ \ \ \geq c_3(a_3(G)-a_3(K_4))\lambda^3+c_4(a_4(G)-a_4(K_4))\lambda^4+c_5(a_5(G)-a_5(K_4))\lambda^5-\varepsilon.$$
Since $c_3,c_5>0$ and $c_4<0$ we have that
$$\frac{1}{v(G)}\ln M_G(\lambda)-\frac{1}{4}\ln M_{K_4}(\lambda)\geq c_3(3-3\rho_3)\lambda^3+c_4(51-48\rho_3-4\rho_4)\lambda^4-\varepsilon.$$
Here
$$c_3(3-3\rho_3)\lambda^3+c_4(51-48\rho_3-4\rho_4)\lambda^4-\varepsilon=\frac{3}{2}c_3\lambda^3+27c_4\lambda^4-\varepsilon+\left(\frac{1}{2}-\rho_3\right)(3c_3\lambda^3+48c_4\lambda^4)-4c_4\rho_4\lambda^4.$$
Note that $3c_3\lambda^3+48c_4\lambda^4>0$ since $\frac{3}{2}c_3\lambda^3+27c_4\lambda^4-\varepsilon>0$ implies that $\frac{3}{2}c_3\lambda^3+27c_4\lambda^4>0$, and consequently
$$\lambda<-\frac{1}{18}\frac{c_3}{c_4}<-\frac{1}{16}\frac{c_3}{c_4}.$$
Since $\rho_3\leq 1/2$ we get that all terms are non-negative by the assumption on $\lambda$, whence
$$\frac{1}{v(G)}\ln M_G(\lambda)-\frac{1}{4}\ln M_{K_4}(\lambda)>0.$$
\end{proof}~\\
Theorem~\ref{main} gives that for $d=3$ we get that for $0<\lambda<\frac{1}{12^2}=0.006944444$ we have
$$\frac{1}{v(G)}\ln M_G(\lambda)\geq \frac{1}{4}\ln M_{K_4}(\lambda).$$
For a given $A$ let $P^{(k)}_A(x)$ be the best approximation of $\ln(1+x)$ on the interval $[0,A]$ with a degree $k$ polynomial with respect to the sup norm. This polynomial can be computed by Remez's algorithm. We will always choose $k=4$. The following table and the polynomials after the table give the necessary information to prove that
$$\frac{1}{v(G)}\ln M_G(\lambda)\geq \frac{1}{4}\ln M_{K_4}(\lambda)$$
holds true for every $3$--regular graph $G$ and $0<\lambda<0.3575$. 

\begin{table}[h!]
\begin{tabular}{|l|l|l|l|}
\hline
$A$    & $\lambda_{\min}(A)$    & $\lambda_{\max}(A)$  & $A/8$ \\ \hline
$0.2$ & $0.005568811878$ & $0.1034074863$ & $0.025$ \\ \hline
$0.5$ & $0.02277925697$  & $0.1461209526$ & $0.0625$ \\ \hline
$0.9$ & $0.05462386679$ & $0.1999613734$ & $0.1125$ \\ \hline
$1.4$ & $0.1046154581$ & $0.2614267492$ & $0.175$ \\ \hline
$1.8$  & $0.1519465525$ & $0.3051509460$ & $0.225$                         \\ \hline
$2.3$  & $0.2219072063$ & $0.3504592779$ & $0.2875$                        \\ \hline
$2.6$  & $0.2731604235$ & $0.3689208685$ & $0.325$                         \\ \hline
$2.8$  & $0.3175642389$ & $0.3711748270$ & $0.35$                          \\ \hline
$2.87$ & $0.3425328478$ & $0.3625667941$ & $0.35875$                       \\ \hline
\end{tabular}
\end{table}

As we can see from the table $A/8\leq \lambda_{\max}(A)$, so it is the intervals $(\lambda_{\min}(A),A/8)$ that we really use.

\begin{align*}
P^{(4)}_{0.2}&=0.7846422000\cdot 10^{-7}+0.9999797421x-0.4991602677x^2\\
&+0.3209653251x^3-0.1724127778x^4\\
P^{(4)}_{0.5}&=0.000004233531000+0.9995443916x-0.4920792546x^2\\
&+0.2833215022x^3-0.1073748605x^4\\
P^{(4)}_{0.9}&=0.00004150761800+0.9974003648x-0.4734793761x^2\\\
&+0.2322928691x^3-0.06357668676x^4\\
P^{(4)}_{1.4}&=0.0001918409080+0.9918765007x-0.4434449266x^2\\\
&+0.1814493104x^3-0.03703844684x^4\\
P^{(4)}_{1.8}&=0.0004241866100+0.9855259927x-0.4182177037x^2\\&+0.1506903343x^3-0.02562265780x^4\\
P^{(4)}_{2.3}&=0.00087040785+0.9758164899x-0.3877699674x^2\\
&+0.1214998746x^3-0.01712337080x^4\\
    P^{(4)}_{2.6}&=0.00122184600+0.9693068092x-0.3705662707x^2\\
&+0.1076897348x^3-0.01377390366x^4\\
P^{(4)}_{2.8}&=0.001490122300+0.9647567731x-0.3596083119x^2\\
&+0.09969352541x^3-0.01201389471x^4\\
P^{(4)}_{2.87}&=0.00159028909+0.9631313107x-0.3558721070x^2\\
&+0.09709474935x^3-0.01146912787x^4
\end{align*}

\newpage

\section{Covers and  necklaces} \label{covers and necklaces}

In this section we study special covers of graphs. Our goal is to show that for odd $d$ there are graphs that beats $K_{d+1}$ for large $\lambda$.

\begin{Def} A $k$-cover (or $k$-lift) $H$ of a graph $G$ is defined as follows. The vertex set of  $H$ is $V(H)=V(G)\times \{0,1,\dots, k-1\}$, and if $(u,v)\in E(G)$, then we choose a perfect matching between the vertices $(u,i)$ and $(v,j)$ for $0\leq i,j\leq k-1$. If $(u,v)\notin E(G)$, then there are no edges between $(u,i)$ and $(v,j)$ for $0\leq i,j\leq k-1$. 
\end{Def}

Among $k$-covers we study special ones that we will call necklace covers.

\begin{Def} Let $G$ be a graph and $(u,v)\in E(G)$, and for a $k\geq 2$ let us consider the following $k$-cover. If $(x,y)\neq (u,v)$ and $(x,y) \in E(G)$, then let us choose the perfect matching $((x,i),(y,i))$ for $i\in  \{0,1,\dots, k-1\}$, and for the edge $(u,v)$ let us choose the perfect matching $((u,i),(v,i+1))$, where $i\in  \{0,1,\dots, k-1\}$ and $(v,k)=(v,0)$. We will denote this graph by $\mathrm{G_{uv}N}_k$, and call it the necklace $k$-cover of the graph $G$ with respect to the edge $(u,v)$.
\end{Def}

The following picture depicts $(K_4)_{uv}N_4$, and explains where the word 'necklace' comes from.

\begin{center}
\includegraphics[scale=0.4]{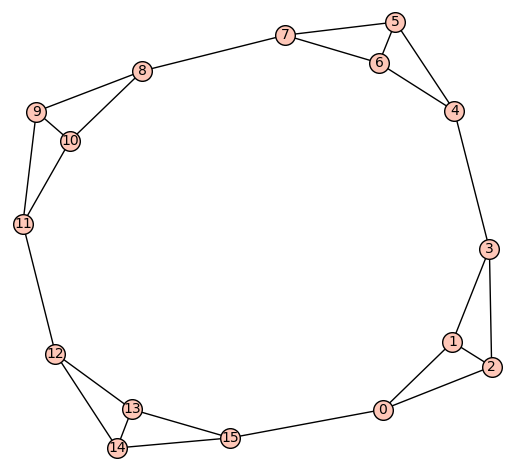}
\end{center}

\begin{Th} \label{transfer_cover} Let $G$ be a graph and $e=(u,v)\in E(G)$. Let $\mathrm{G_{uv}N}_k$ be  the necklace $k$-cover of the graph $G$ with respect to the edge $(u,v)$. Let $A(G_{uv})$ be the following matrix:
$$A(G_{uv}):=\left(\begin{array}{cc} 
M(G-e,\lambda) & \sqrt{\lambda}M(G-u,\lambda) \\
\sqrt{\lambda}M(G-v,\lambda) & \lambda M(G-\{u,v\},\lambda) 
\end{array}
\right).$$
Let $\vartheta_1$ and $\vartheta_2$ be the two eigenvalues of $A$. Then
$$M(\mathrm{G_{uv}N}_k,\lambda)=\vartheta_1^k+\vartheta_2^k.$$
\end{Th}

\begin{proof} For sake of simplicity let the rows and columns of $A(G_{uv})$ be indexed by $0$ and $1$, that is, $A(G_{uv})=\left(\begin{array}{cc} 
a_{00} & a_{01} \\
a_{10} & a_{11} \\ 
\end{array}\right)$.
Let $e_1,e_2,\dots, e_k$ be the lifts of the edge $(u,v)$. For a matching $M\in \mathcal{M}(G)$ let $\chi_M$ be its characteristic vector, that is, 
$$\chi_M(e)=\left\{ \begin{array}{ll} 1 & \mbox{if}\ e\in M, \\ 0 & \mbox{if}\ e\notin M. \end{array} \right. $$ 
For a sequence $(s_1,s_2,\dots ,s_k)\in \{0,1\}^k$, let us consider those matchings $M$ such that $\chi_M(e_j)=s_j$.  Then the contribution of these matchings to $M(\mathrm{G_{uv}N}_k,\lambda)$ is exactly $a_{s_1,s_2}a_{s_2,s_3}\dots a_{s_k,s_1}$. The reason is the following: suppose for instance that $s_1=0$ and $s_2=1$, then from the first copy of $G-e$ the vertex $u$ is covered by the edge $e_2$, but otherwise we can choose any matching of $G-u$. The $\sqrt{\lambda}$ factor comes from taking the half weight of $e_2$ into account, the other half will come from $a_{s2,s3}$. The same reasoning applies for other values of $s_1$ and $s_2$. Hence
$$M(\mathrm{G_{uv}N}_k,\lambda)=\sum_{\underline{s}\in \{0,1\}^k}a_{s_1,s_2}a_{s_2,s_3}\dots a_{s_k,s_1}.$$
Clearly,
$$\sum_{\underline{s}\in \{0,1\}^k}a_{s_1,s_2}a_{s_2,s_3}\dots a_{s_k,s_1}=\mathrm{Tr}A(G_{uv})^k=\vartheta_1^k+\vartheta_2^k.$$
\end{proof}

\begin{Cor} \label{cor_transfer_cover} Let $\lambda\geq 0$. Let $G$ be a graph and $e=(u,v)\in E(G)$. Let $\mathrm{G_{uv}N}_k$ be  the necklace $k$-cover of the graph $G$ with respect to the edge $(u,v)$. Let
$$d(G_{uv})=M(G-e,\lambda)M(G-\{u,v\},\lambda)-M(G-u,\lambda)M(G-v,\lambda).$$
Then

\noindent (a) For  $k\geq 2$ we have $M(\mathrm{G_{uv}N}_k,\lambda)>M(G,\lambda)^k$ if $d(G_{uv})<0$.
\medskip

\noindent (b) For  $k\geq 2$ we have $M(\mathrm{G_{uv}N}_k,\lambda)=M(G,\lambda)^k$ if $d(G_{uv})=0$.
\medskip

\noindent (c) For  $k\geq 2$ we have $M(\mathrm{G_{uv}N}_k,\lambda)<M(G,\lambda)^k$ if $d(G_{uv})>0$.
\end{Cor}

\begin{proof} Let $\vartheta_1$ and $\vartheta_2$ be the  eigenvalues of $A(G_{uv})$. These are real numbers since the eigenvalues of a $2\times 2$ matrix with non-negative real entries are always real. Note that
$$\vartheta_1+\vartheta_2=\mathrm{Tr}(A(G_{uv}))=M(G-e,\lambda)+\lambda M(G-\{u,v\},\lambda)=M(G,\lambda).$$
This is positive, so if $\vartheta_1$ is the largest eigenvalue, then $\vartheta_1>0$. 
Furthermore, we have $\vartheta_1\vartheta_2=\det(A(G_{uv}))=\lambda d(G_{uv})$. 
Now if $d(G_{uv})<0$, then $\vartheta_2<0$, and 
$$M(\mathrm{G_{uv}N}_k,\lambda)=\vartheta_1^k+\vartheta_2^k=(M(G,\lambda)-\vartheta_2)^k-\vartheta_2^k>M(G,\lambda)^k.$$
If $d(G_{uv})=0$, then $\vartheta_2=0$, and
$$M(\mathrm{G_{uv}N}_k,\lambda)=\vartheta_1^k+\vartheta_2^k=\vartheta_1^k=M(G,\lambda)^k.$$
If $d(G_{uv})>0$, then $\vartheta_2>0$, and
$$M(\mathrm{G_{uv}N}_k,\lambda)=\vartheta_1^k+\vartheta_2^k<(\vartheta_1+\vartheta_2)^k=M(G,\lambda)^k.$$

\end{proof}

Diamond is the unique simple graph with $4$ vertices and $5$ edges. A diamond necklace consisting of $k$ diamonds is the graph obtained from $k$ diamonds by connecting the degree $2$ vertices cyclically to make it a $3$--regular graph. Let us denote it by $\mathrm{DN}_k$ instead of $(K_4)_{uv}N_k$. The following is now an easy corollary of our previous discussion.

\begin{Th}
(a) For  $k\geq 2$ we have 
$$M(K_4,\lambda)^k< M(\mathrm{DN}_k,\lambda)$$
for $\lambda<1$.\\
(b) For  $k\geq 2$ we have 
$$M(K_4,\lambda)^k> M(\mathrm{DN}_k,\lambda)$$
for $\lambda>1$.\\
(c) For  $k\geq 2$ we have 
$$M(K_4,1)^k= M(\mathrm{DN}_k,1).$$
\end{Th} 

\begin{proof} This is a special case of Corollary~\ref{cor_transfer_cover}. Note that
$$d((K_4)_{uv})=(1+5\lambda+2\lambda^2)\lambda^2(1+\lambda)-\lambda^2(1+3\lambda)^2=2(\lambda^4-\lambda^3).$$
Hence the sign of $d((K_4)_{uv})$ changes exactly at $\lambda=1$. 
\end{proof}

In general let $c_d$ be the smallest positive root (if it exists) of the polynomial 
$$M(K_{d+1}-e,\lambda)M(K_{d+1}-\{u,v\},\lambda)-M(K_{d+1}-u,\lambda)M(GK_{d+1}-v,\lambda),$$
where $e=(u,v)$ is some edge of $K_{d+1}$. This is, of course,
$$P_d(\lambda):=M(K_{d+1}-e,\lambda)M(K_{d-1},\lambda)-M(K_{d},\lambda)^2.$$
We will actually show that if $d$ is odd, then $P_d(\lambda)$ has exactly one positive root, and if $d$ is even, then it has no positive root. In fact, we will show that $P_d(\lambda)$ has a very special form: if $d$ is even, then all of its coefficients are negative, if $d$ is odd, then all except the top coefficient is negative.

\begin{longtable}{|c|c|}
\hline
$d$ & $P_d(\lambda)$                                                                                                   \\ \hline
\endfirsthead
\endhead
$3$ & $-2\lambda^2(1-\lambda)$                                                                                         \\ \hline
$4$ & $-3\lambda^2(1+3\lambda^2)$                                                                                      \\ \hline
$5$ & $-4\lambda^2(1+3\lambda+9\lambda^2-9\lambda^3)$                                                                  \\ \hline
$6$ & $-5\lambda^2(1+8\lambda+30\lambda^2+45\lambda^4)$                                                                \\ \hline
$7$ & $-6\lambda^2(1+15\lambda+90\lambda^2+150\lambda^3+225\lambda^4-225\lambda^5)$                                    \\ \hline
$8$ & $-7\lambda^2(1+24\lambda+225\lambda^2+840\lambda^3+1575\lambda^4+1575\lambda^6)$                                 \\ \hline
$9$ & $-8\lambda^2(1+35\lambda+483\lambda^2+3045\lambda^3+9555\lambda^4+11025\lambda^5+11025\lambda^6-11025\lambda^7)$ \\ \hline
\end{longtable}

In what follows we will frequently use the recursion
\begin{align} \label{recursion}
M(G,\lambda)=M(G-u,\lambda)+\lambda \sum_{v\in N_G(u)}M(G-\{u,v\},\lambda)
\end{align}
that simply relies on the fact that a matching either does not cover the vertex $u$, or it covers $u$ together with one of its neighbors. Applying it to $M(K_{n+1}-e,\lambda)$ and one of the end vertex of $e$ we get that
$$M(K_{n+1}-e,\lambda)=M(K_n,\lambda)+(n-1)\lambda M(K_{n-1},\lambda).$$
From now on we use the notation $q_n:=M(K_n,\lambda)$. Note that it has an explicit form:
$$q_n=M(K_n,\lambda)=\sum_{0\leq r\leq n/2}\frac{n!}{2^rr!(n-2r)!}\lambda^r,$$
but we will use this explicit form only once. Instead we mostly use the recursion
$$q_n=q_{n-1}+(n-1)\lambda q_{n-2}$$
that is a direct consequence of the identity~\ref{recursion}.

\begin{Lemma}
Let $Q_d(\lambda)=q_dq_{d-2}-q_{d-1}^2$. \\
(a) Then $P_d(\lambda)=-(d-1)\lambda Q_d(\lambda)$.\\
(b) $Q_d(\lambda)$ satisfies the following two recursions. If $d\geq 5$, then
$$Q_d(\lambda)=\lambda q_{d-2}(q_{d-3}-\lambda q_{d-4})+(d-2)^2\lambda^2Q_{d-2}(\lambda).$$
$$Q_d(\lambda)=\lambda q_{d-3}(q_{d-2}-\lambda q_{d-3})+(d-1)(d-3)\lambda^2Q_{d-2}(\lambda).$$
(c) If $d$ is odd, then all coefficients except the top coefficient of $Q_d(\lambda)$ is positive. The top coefficient is negative. If $d$ is even, then all coefficients of $Q_d(\lambda)$ is positive.\\
(d) If $d$ is even, then $Q_d(\lambda)$ has no positive root. If $d$ is odd, then $Q_d(\lambda)$ has a unique positive root $c_d$. We have $c_3<c_5<c_7<\dots $ and $\lim_{d\to \infty}c_d=\infty$.
\end{Lemma}

\begin{proof}
(a) We have
\begin{align*}
P_d(\lambda)&=M(K_{d+1}-e,\lambda)M(K_{d-1},\lambda)-M(K_d,\lambda)^2\\
&=(q_d+(d-1)\lambda q_{d-1})q_{d-1}-q_d^2\\
&=(q_d+(d-1)\lambda q_{d-1})q_{d-1}-q_d(q_{d-1}+(d-1)\lambda q_{d-2})\\
&=-(d-1)\lambda (q_dq_{d-2}-q_{d-1}^2)\\
&=-(d-1)\lambda Q_d(\lambda).
\end{align*}
(b) We have
\begin{align*}
Q_d(\lambda)&=q_dq_{d-2}-q_{d-1}^2\\
&=(q_{d-1}+(d-1)\lambda q_{d-2})q_{d-2}-q_{d-1}(q_{d-2}+(d-2)\lambda q_{d-3})\\
&=(d-1)\lambda q_{d-2}^2-(d-2)\lambda q_{d-1}q_{d-3}\\
&=(d-1)\lambda q_{d-2}(q_{d-3}+(d-3)\lambda q_{d-4})-(d-2)\lambda (q_{d-2}+(d-2)\lambda q_{d-3})q_{d-3}\\
&=\lambda q_{d-2}q_{d-3}+(d-1)(d-3)\lambda^2q_{d-2}q_{d-4}-(d-2)^2\lambda^2q_{d-3}^2\\
&=\lambda q_{d-2}(q_{d-3}-\lambda q_{d-4})+(d-2)^2\lambda^2(q_{d-2}q_{d-4}-q_{d-3}^2)\\
&=q_{d-2}(q_{d-3}-\lambda q_{d-4})+(d-2)^2\lambda^2Q_{d-2}(\lambda).
\end{align*}
Alternatively, we can write the last two lines as follows:
\begin{align*}
Q_d(\lambda)&=\lambda q_{d-2}q_{d-3}+(d-1)(d-3)\lambda^2q_{d-2}q_{d-4}-(d-2)^2\lambda^2q_{d-3}^2\\
&=\lambda (q_{d-2}-\lambda q_{d-3})q_{d-3}+(d-1)(d-3)\lambda^2(q_{d-2}q_{d-4}-q_{d-3}^2)\\
&=\lambda q_{d-3}(q_{d-2}-\lambda q_{d-3})+(d-1)(d-3)\lambda^2Q_{d-2}(\lambda).
\end{align*}
(c) First we prove that if $n$ is even, then $q_n-\lambda q_{n-1}$ is a polynomial of degree $n/2-1$ with only positive coefficients. Indeed, by the explicit form of $q_n$ we get that the coefficient of $\lambda^r$ is $q_n-\lambda q_{n-1}$ is
$$\frac{n!}{2^rr!(n-2r)!}-\frac{(n-1)!}{2^{r-1}(r-1)!(n-2r+1)!}=\frac{(n-1)!}{2^{r-1}(r-1)!(n-2r)!}\left(\frac{n}{2r}-\frac{1}{n-2r+1}\right).$$
If $r=n/2$, then this is $0$. If $r<n/2$ this is positive since $n/(2r)>1>\frac{1}{n-2r+1}$.

Now we are ready to prove the statements of part (c). We prove them by induction. We can assume that $d\geq 5$ since the table shows that the statement is true for $d\leq 4$. First assume that $d$ is even. Then we use the recursion formula 
$$Q_d(\lambda)=\lambda q_{d-3}(q_{d-2}-\lambda q_{d-3})+(d-1)(d-3)\lambda^2Q_{d-2}(\lambda).$$
By induction $Q_{d-2}(\lambda)$ has only positive coefficients. Since $d-2$ is even we get that $q_{d-2}-\lambda q_{d-3}$ has also only positive coefficients. So in this case we are done.
If $d$ is odd, then we use the recursion formula 
$$Q_d(\lambda)=\lambda q_{d-2}(q_{d-3}-\lambda q_{d-4})+(d-2)^2\lambda^2Q_{d-2}(\lambda).$$
By induction $Q_{d-2}(\lambda)$ has only positive coefficients except the top one which is negative. Since $d-3$ is even we get that $q_{d-3}-\lambda q_{d-4}$ has only positive coefficients. Note that $\deg(q_{d-2})=(d-3)/2$, $\deg(q_{d-3}-\lambda q_{d-4})=(d-5)/2$ while by induction $\deg(Q_{d-2}(\lambda))=d-3$, so $\deg((d-1)(d-3)\lambda^2Q_{d-2}(\lambda))=d-1$ and $\deg(\lambda q_{d-2}(q_{d-3}-\lambda q_{d-4}))=d-3$, thus the top coefficient will not be affected. (We can also see that it is actually \\ $-((d-2)(d-4)\dots 1)^2$.\\
(d) If $d$ is even, then all coefficients of $Q_d(\lambda)$ are positive, so it cannot have a positive real root. If $d$ is odd, then the form of the polynomial immediately implies that it has a real root, since the polynomial is positive for very small $\lambda$ and is negative for very large $\lambda$ since the top coefficient is negative. We also show that such a polynomial cannot have two positive zeros. Suppose for contradiction that a polynomial of the form
$$p(x)=\sum_{j=0}^{r-1}a_jx^j-a_rx^r$$
have two positive roots, $\gamma_1$ and $\gamma_2$ and $a_j\geq 0$ for $j=0,\dots ,r$. We can assume that $\gamma_2>\gamma_1$. Then
$$a_r\gamma_2^r=a_r\gamma_1^r\left(\frac{\gamma_2}{\gamma_1}\right)^r=\left(\sum_{j=0}^{r-1}a_j\gamma_1^j\right)\left(\frac{\gamma_2}{\gamma_1}\right)^r=\sum_{j=0}^{r-1}a_j\gamma_2^j\left(\frac{\gamma_2}{\gamma_1}\right)^{r-j}\geq \frac{\gamma_2}{\gamma_1}\sum_{j=0}^{r-1}a_j\gamma_2^j=\frac{\gamma_2}{\gamma_1}a_r\gamma_2^r.$$
This is, of course, contradiction. This shows that for odd $d$ the polynomial $Q_d(\lambda)$ has exactly one positive root that we can denote by $c_d$. This also shows that for $0<\lambda<c_d$ we get that $Q_d(\lambda)>0$, and for $\lambda>c_d$ we get that $Q_d(\lambda)<0$. Now let us evaluate $Q_d(\lambda)$ at $c_{d-2}$. Then
$$Q_d(c_{d-2})=\lambda q_{d-2}(q_{d-3}-\lambda q_{d-4})\Big|_{\lambda=c_{d-2}}+(d-2)^2\lambda^2Q_{d-2}(c_{d-2})=\lambda q_{d-2}(q_{d-3}-\lambda q_{d-4})\Big|_{\lambda=c_{d-2}}>0$$
since all terms of $\lambda q_{d-2}(q_{d-3}-\lambda q_{d-4})$ are positive. Hence $c_d>c_{d-2}$. To prove that $\lim_{d\to \infty}c_d=\infty$ we need to do a little additional computation. We have seen that it is quite easy to prove that the coefficient of $\lambda^{d-1}$ in $Q_{d}(\lambda)$ is $-((d-2)!!)^2$, where $(d-2)!!=(d-2)(d-4)\dots 3\cdot 1$. From the recursion it is also easy to see that the coefficient of 
$\lambda^{d-2}$ in $Q_{d}(\lambda)$ is $((d-2)!!)^2$. It is a bit harder to see, but one can check even by direct computation is that the coefficient of 
$\lambda^{d-3}$ in $Q_{d}(\lambda)$ is $\frac{d-3}{6}((d-2)!!)^2$. Indeed, if $[\lambda^r]p$ denotes the coefficient of $\lambda^r$ in $p$, then
\begin{align*}
[\lambda^{d-3}]Q_d(\lambda)&=[\lambda^{(d-1)/2}]q_d\cdot [\lambda^{(d-5)/2}]q_{d-2}+[\lambda^{(d-3)/2}]q_d\cdot [\lambda^{(d-3)/2}]q_{d-2}\\
&-2[\lambda^{(d-1)/2}]q_{d-1}\cdot [\lambda^{(d-5)/2}]q_{d-1}
-[\lambda^{(d-3)/2}]q_{d-1}\cdot [\lambda^{(d-3)/2}]q_{d-1}\\
&=\frac{d!}{2^{(d-1)/2}((d-1)/2)!}\cdot \frac{(d-2)!}{2^{(d-5)/2}((d-5)/2)!3!}\\
&+\frac{d!}{2^{(d-3)/2}((d-3)/2)!3!}\cdot \frac{(d-2)!}{2^{(d-3)/2}((d-3)/2)!}\\
&-2\frac{(d-1)!}{2^{(d-1)/2}((d-1)/2)!}\cdot \frac{(d-1)!}{2^{(d-5)/2}((d-5)/2)!4!}\\
&-\frac{(d-1)!}{2^{(d-3)/2}((d-3)/2)!2!}\cdot \frac{(d-1)!}{2^{(d-3)/2}((d-3)/2)!2!}\\
&=\frac{(d-1)!(d-2)!}{2^{d-3}3!((d-5)/2)!((d-1)/2)!}\\
&=\frac{d-3}{6}((d-2)!!)^2
\end{align*}

Then by keeping just the terms corresponding to $\lambda^{d-3}$ and $\lambda^{d-1}$, and forget the remaining positive terms we get that
$$0=Q_d(c_d)\geq \frac{d-3}{6}((d-2)!!)^2c_d^{d-3}-((d-2)!!)^2c_d^{d-1}$$
whence $c_d^2\geq \frac{d-3}{6}$, that is, $c_d\geq \sqrt{\frac{d-3}{6}}$. Hence $\lim_{d\to \infty}c_d=\infty$.
\end{proof}

\begin{table}[h!]
\begin{tabular}{|c|c|} \hline 
 $d$    &  $c_d$  \\ \hline 
 $3$    &   $1$   \\ \hline 
 $5$    &   $1.317124345$   \\ \hline 
 $7$    &   $1.593204592$   \\ \hline 
 $9$    &   $1.844705431$   \\ \hline 
 \end{tabular}
 \end{table}
 
 Let us summarize what happened so far. We introduced the polynomial 
 $$P_d(\lambda)=M(K_{d+1}-e,\lambda)M(K_{d-1},\lambda)-M(K_{d},\lambda)^2,$$
 and we showed that if $d$ is odd, then it has a unique positive root $c_d$. By our previous argument on necklaces we see that if $\lambda<c_d$, then for all $k>1$ we have
 $$\frac{1}{v(K_{d+1})}\ln M(K_{d+1},\lambda)<\frac{1}{v((K_{d+1})_{uv}N_k)}\ln M((K_{d+1})_{uv}N_k,\lambda).$$
 If $\lambda=c_d$, then for all $k$ we have
 $$\frac{1}{v(K_{d+1})}\ln M(K_{d+1},\lambda)=\frac{1}{v((K_{d+1})_{uv}N_k)}\ln M((K_{d+1})_{uv}N_k,\lambda).$$
 If $\lambda>c_d$, then for all $k>1$ we have 
 $$\frac{1}{v(K_{d+1})}\ln M(K_{d+1},\lambda)>\frac{1}{v((K_{d+1})_{uv}N_k)}\ln M((K_{d+1})_{uv}N_k,\lambda).$$
 So for $\lambda>c_d$ the graph $K_{d+1}$ is never the minimizer graph. 
 
\section{The case of even $d$} \label{sec: even d}

In this section we prove Theorem~\ref{even d}. The strategy is very simple: we show that a graph that has no connected component isomorphic to $K_{d+1}$ necessarily have a large 
matching. Note that the largest matching of $K_{d+1}$ has size $d/2$, that is, its size is $\frac{d}{2(d+1)}v(K_{d+1})$.
First, we show that there is a constant  $r_d$ larger than $\frac{d}{2(d+1)}$  such that if $G$ is a $d$--regular graph without connected components isomorphic to $K_{d+1}$, then $G$ has a matching of size at least $r_dv(G)$.

\begin{Lemma}
Let $d$ be even. Suppose that $G$ is a $d$--regular graph without connected component isomorphic to $K_{d+1}$. Then the size of the largest matching is at least $\frac{d+2}{2(d+3)}v(G)$.
\end{Lemma}

The following proof is inspired by the proof of Flaxman and Hoory \cite{flaxman2007maximum}.

\begin{proof}
We will use the description of matching polytope due to Edmonds \cite{edmonds1965maximum}, for a short proof see \cite{schrijver1983short}. The matching polytope is defined as follows: let us consider the vector space $\mathbb{R}^{E(G)}$ and for each matching $M$ let us associate its characteristic  vector:
$$\chi_{M}(e)=\left\{ \begin{array}{cl}
1 & \mbox{if}\ e\in M, \\
0 & \mbox{if}\ e\notin M.
\end{array}\right.
$$
Then the matching polytope of the graph $G$ is simply the convex hull of the vectors $\chi_M$:
$$\mathrm{MP}(G)=\mathrm{conv}\left\{\chi_M\ \ \mbox{for}\ M\in \mathcal{M}(G)\right\}.$$
Edmonds \cite{edmonds1965maximum} proved an alternative characterization of the matching polytope. He proved that $x\in \mathrm{MP}(G)$ if and only if the following conditions hold true:\\
(i) $x_e\geq 0$, that is, the vector is non-negative\\
(ii) for all $v\in V(G)$ we have $\sum_{e: v\in e}x_e\leq 1$\\
(iii) for all set $S\subseteq V(G)$ such that $|S|$ is odd we have
$$\sum_{e\in E(S)}x_e\leq \frac{|S|-1}{2},$$
where $E(S)$ is the set of edges induced by the vertices in $S$.
It is easy to see that the characteristic vectors of the matchings satisfy these conditions, and so their convex combinations. The converse is the more difficult part of this theorem, and this is what we will use. We will show that the vector $x$ that is constant $\frac{d+2}{d(d+3)}$ on every edge satisfies the condition. It is clearly non-negative, and for every vertex $v$ we have
$$\sum_{e: v\in e}x_e=\frac{d(d+2)}{d(d+3)}\leq 1.$$
To check condition (iii) we distinguish three cases: (I) $|S|\leq d-1$, (II) $|S|=d+1$ and (III) $|S|\geq d+3$. In case (I) we have
$$\sum_{e\in E(S)}x_e\leq \binom{|S|}{2}\frac{d+2}{2(d+3)}=
\frac{|S|-1}{2}\frac{(d+2)|S|}{d(d+3)}\leq \frac{|S|-1}{2}\frac{(d+2)(d-1)}{d(d+3)}\leq \frac{|S|-1}{2}.$$
In case (II) we use the condition of the theorem that $G$ has no connected component isomorphic to $K_{d+1}$, so every $d+1$ vertices induces at most $\binom{d+1}{2}-1=\frac{d^2+d-2}{2}$ edges.
$$\sum_{e\in E(S)}x_e=|E(S)|\frac{d+2}{d(d+3)}\leq \frac{d^2+d-2}{2}\cdot \frac{d+2}{d(d+3)}< \frac{d}{2}=\frac{|S|-1}{2}.$$
Here $(d^2+d-2)(d+2)=d^3+3d^2-4<d^3+3d^2$.
In case (III) we simply use the fact that $S$ can only induce at most $d|S|/2$ edges.
$$\sum_{e\in E(S)}x_e=|E(S)|\frac{d+2}{d(d+3)}\leq \frac{d|S|}{2}\cdot \frac{d+2}{d(d+3)}\leq \frac{|S|-1}{2},$$
where the last inequality is true since $ \frac{|S|}{|S|-1}\leq \frac{d+3}{d+2}$ since $|S|\geq d+3$.
Hence the vector $x$ is the matching polytope, that is, we can find matchings $M\in \mathcal{G}$ and non-negative $\alpha_M$ such that $x=\sum \alpha_M\chi_M$ and $\sum \alpha_M=1$.

Now for a $y\in \mathbb{R}^{E(G)}$ let $|y|=\sum_{e\in E(G)}y_e$. Then
$$|x|=\frac{dv(G)}{2}\cdot \frac{d+2}{d(d+3)}=\frac{d+2}{2(d+3)}v(G).$$
On the other hand,
$$|x|=\left|\sum \alpha_M\chi_M\right|=\sum \alpha_M |\chi_M|=\sum \alpha_M|M|\leq \max |M|.$$
Hence there must be a matching of size at least $\frac{d+2}{2(d+3)}v(G)$.

\end{proof}

Now we are ready to prove Theorem~\ref{even d}.

\begin{proof}[Proof of Theorem~\ref{even d}]
We can assume that $\lambda>1$. Then
$$M_{K_{d+1}}(\lambda)\leq \lambda^{d/2}M_{K_{d+1}}(1)\leq (d+1)^{d+1}\lambda^{d/2},$$
where the upper bound $M_G(1)\leq (d+1)^{d+1}$ is very generous, and follows from the fact that at each vertex we can choose at most one edge. On the other hand, if $\nu(G)$ denotes the size of the largest matching, then
$$M_G(\lambda)\geq \lambda^{\nu(G)}\geq \lambda^{\frac{d+2}{2(d+3)}v(G)}$$
by the previous lemma.
Hence
$$\frac{1}{v(K_{d+1})}\ln M_{K_{d+1}}(\lambda)\leq \ln(d+1)+\frac{d}{2(d+1)}\ln (\lambda),$$
while 
$$\frac{1}{v(G)}\ln M_G(\lambda)\geq \frac{d+2}{2(d+3)}\ln (\lambda).$$
Hence if 
$$\frac{d+2}{2(d+3)}\ln (\lambda)\geq \ln(d+1)+\frac{d}{2(d+1)}\ln (\lambda),$$
then we are done. This is satisfied if 
$$\lambda\geq \exp((d+1)(d+3)\ln(d+1)).$$

\end{proof}

\noindent \textbf{Acknowledgment.} The second author thanks Ferenc Bencs and Will Perkins for the discussions on the topic of this paper.

\bibliographystyle{siamnodash}

\bibliography{biblio_matchings_regular_graphs}

\begin{thebibliography}{10}

\bibitem{csikvari2014lower}
{\sc P.~Csikv{\'a}ri}, {\em {Lower matching conjecture, and a new proof of
  Schrijver's and Gurvits's theorems}}, Journal of the European Mathematical
  Society, {\bf 19} (2017), pp.~1811--1844.

\bibitem{davies2020proof}
{\sc E.~Davies, M.~Jenssen, and W.~Perkins}, {\em A proof of the upper matching
  conjecture for large graphs}, arXiv preprint arXiv:2004.06695,  (2020).

\bibitem{davies2017independent}
{\sc E.~Davies, M.~Jenssen, W.~Perkins, and B.~Roberts}, {\em Independent sets,
  matchings, and occupancy fractions}, Journal of the London Mathematical
  Society, {\bf 96} (2017), pp.~47--66.

\bibitem{edmonds1965maximum}
{\sc J.~Edmonds}, {\em Maximum matching and a polyhedron with 0, 1-vertices},
  Journal of research of the National Bureau of Standards B, {\bf 69} (1965),
  pp.~55--56.

\bibitem{flaxman2007maximum}
{\sc A.~D. Flaxman and S.~Hoory}, {\em Maximum matchings in regular graphs of
  high girth}, The Electronic Journal of Combinatorics,  (2007), pp.~N1--N1.

\bibitem{friedland2008number}
{\sc S.~Friedland, E.~Krop, and K.~Markstr{\"o}m}, {\em On the number of
  matchings in regular graphs}, the electronic journal of combinatorics,
  (2008), pp.~R110--R110.

\bibitem{godsil1993algebraic}
{\sc C.~Godsil}, {\em Algebraic combinatorics}, vol.~6, CRC Press, 1993.

\bibitem{gurvits2011unleashing}
{\sc L.~Gurvits}, {\em {Unleashing the power of Schrijver's permanental
  inequality with the help of the Bethe approximation}}, arXiv preprint
  arXiv:1106.2844,  (2011).

\bibitem{harangi2013density}
{\sc V.~Harangi}, {\em On the density of triangles and squares in regular
  finite and unimodular random graphs}, Combinatorica, {\bf 33} (2013),
  pp.~531--548.

\bibitem{heilmann1972theory}
{\sc O.~J. Heilmann and E.~H. Lieb}, {\em Theory of monomer-dimer systems}, in
  Statistical Mechanics, Springer, 1972, pp.~45--87.

\bibitem{mckay1998maximising}
{\sc B.~D. McKay and I.~M. Wanless}, {\em Maximising the permanent of (0,
  1)-matrices and the number of extensions of latin rectangles}, The Electronic
  Journal of Combinatorics,  (1998), pp.~R11--R11.

\bibitem{schrijver1983short}
{\sc A.~Schrijver}, {\em Short proofs on the matching polyhedron}, Journal of
  Combinatorial Theory, Series B, {\bf 34} (1983), pp.~104--108.

\bibitem{wanless2010counting}
{\sc I.~M. Wanless}, {\em Counting matchings and tree-like walks in regular
  graphs}, Combinatorics, Probability and Computing, {\bf 19} (2010),
  pp.~463--480.

\end{thebibliography}

\end{document}